\documentclass[12pt,reqno]{article}
\usepackage{amsmath,amsthm,amsfonts,amssymb,amscd}
\usepackage[dvips]{graphicx}
\usepackage{psfrag}

\theoremstyle{plain}
\newtheorem{thm}{Theorem}[section]

\newtheorem{prop}[thm]{Proposition}
\newtheorem{clly}[thm]{Corollary}
\newtheorem{lemma}[thm]{Lemma}
\newtheorem{defi}[thm]{Definition}

\newtheorem{maintheorem}{Theorem}

\newcommand{\re}{{\Bbb R}}
\newcommand{\nat}{{\Bbb N}}

\title{Finiteness of attractors and repellers on sectional hyperbolic sets}
\author{A. M. L\'opez B.
        \thanks{
{\em Key words and phrases}:
Attractor, Repeller, Maximal invariant, Sectional-Anosov flow.
This work is partially supported by CAPES, Brazil.}}
\date{}

\begin{document}
\maketitle

\begin{abstract}
We obtain an upper bound for the number of attractors and repellers
that can appear from small perturbations of a sectional hyperbolic set.
This extends results from \cite{a} and \cite{m}.
\end{abstract}


\section{Introduction}

\noindent
A {\em sectional hyperbolic set}  
is a is a partially hyperbolic set whose singularities
are hyperbolic and whose central subbundle is 
sectionally expanding.

The result \cite{m} asserts that for every sectional hyperbolic transitive attracting set
$\Lambda$ of a vector field $X$ on a compact $3$-manifold there are neighborhoods $\mathcal{U}$ of $X$ and $U$ of $\Lambda$
such that the number of attractors in $U$ of a vector field in $\mathcal{U}$ is less than
one plus the number of equilibria of $X$. 
This result was extended later in \cite{ams} by allowing $\Lambda$ to be an attracting set
contained in the nonwandering set (rather than transitive). An extension of \cite{m}
to higher dimensions was recently obtained in \cite{a}.
The present work removes both transitivity and nonwandering hypotheses in order to prove that
for every sectional hyperbolic set
$\Lambda$ of a vector field $X$ on a compact manifold there are neighborhoods $\mathcal{U}$ of $X$, $U$ of $\Lambda$ and a positive integer
$n_0$ such that the number of attractors in $U$ of a vector field in $\mathcal{U}$ is less than $n_0$.
Let us state our result in a precise way.

Consider a compact manifold $M$ of dimension $n\geq 3$ (a {\em compact $n$-manifold} for short)
with a Riemannian structure $\|\cdot\|$.
We denote by $\partial M$ the boundary of $M$.
Let ${\cal X}^1(M)$ be the space
of $C^1$ vector fields in $M$ endowed with the
$C^1$ topology.
Fix $X\in {\cal X}^1(M)$, inwardly
transverse to the boundary $\partial M$ and denotes
by $X_t$ the flow of $X$, $t\in I\!\! R$.
The {\em maximal invariant} set of $X$ is defined by 
$$
M(X)= \displaystyle\bigcap_{t \geq 0} X_t(M).
$$
Notice that $M(X)=M$ in the boundaryless case $\partial M=\emptyset$.
A subset $\Lambda$ is called {\em invariant} if
$X_t(\Lambda)=\Lambda$ for every $t\in I\!\! R$.
We denote by $m(L)$ the minimum norm of a linear
operator $L$, i.e., $m(L)= inf_{v \neq 0} \frac{\left\|Lv\right\|}{\left\|v\right\|} $.

\begin{defi}
\label{d2}
A compact invariant set
$\Lambda$ of $X$
is {\em partially hyperbolic}
if there is a continuous invariant
splitting $T_\Lambda M=E^s\oplus E^c$ 
such that the following properties
hold for some positive constants $C,\lambda$:

\begin{enumerate}
\item
$E^s$ is {\em contracting}, i.e., 
$\mid\mid DX_t(x) \left|_{E^s_x}\right. \mid\mid
\leq Ce^{-\lambda t}$, 
for all $x\in \Lambda$ and $t>0$.
\item
$E^s$ {\em dominates} $E^c$, i.e., 
$\frac{\mid\mid DX_t(x) \left|_{E^s_x}\right. \mid\mid}{m(DX_t(x) \left|_{E^c_x}\right. )}
\leq Ce^{-\lambda t}$, 
for all $x\in \Lambda$ and $t>0$.
\end{enumerate}

We say the central subbundle $E^c_x$ of $\Lambda$ is 
{\em sectionally expanding} if
$$dim(E^c_x) \geq 2 \quad 
and \quad 
\left| det(DX_t(x) \left|_{L_x}\right. ) \right| \geq C^{-1}e^{\lambda t}, \qquad
\forall x \in \Lambda \quad and \quad t > 0  $$ 
for all two-dimensional subspace $L_x$ of $E^c_x$.
Here $det(DX_t(x) \left|_{L_x}\right. )$ denotes
the jacobian of $DX_t(x)$ along $L_x$.
\end{defi}

Recall that a singularity of a vector field is hyperbolic if
the eigenvalues of its linear part
have non zero real part.

\begin{defi}
\label{shs}
A {\em sectional hyperbolic set}
is a partially hyperbolic set whose singularities (if any) are hyperbolic and whose central subbundle is sectionally expanding.
\end{defi}

The $\omega$-limit set of $p\in M$ is the set
$\omega_X(p)$ formed by those $q\in M$ such that $q=\lim_{n\infty}X_{t_n}(p)$ for some
sequence $t_n\to\infty$.
We say that $\Lambda\subset M$ is {\em transitive} if
$\Lambda=\omega_X(p)$ for some $p\in \Lambda$.
We say that $\Lambda$ is {\em singular} if it
contains a singularity; and
{\em attracting}
if $\Lambda=\cap_{t>0}X_t(U)$
for some compact neighborhood $U$ of $\Lambda$.
This neighborhood is called
{\em isolating block} of $\Lambda$.
It is well known that the isolating block $U$ can be chosen to be
positively invariant, namely $X_t(U)\subset U$ for all
$t>0$.
An {\em attractor} is a transitive attracting set.
A {\em repelling} is an attracting for the time reversed vector field $-X$ 
and a {\em repeller} is a transitive repelling set.

With these definitions we can state our main result.

\begin{maintheorem}
\label{thB}
For every sectional hyperbolic set $\Lambda$ of a vector field $X$ on a compact manifold there are
neighborhoods ${\cal U}$ of $X$, 
$U$ of $\Lambda$ and $n_0\in \mathbb{N}$ such that
\begin{center}
$\#\{L\subset U:$ $L$ is an attractor or repeller of $Y\in {\cal U}\} \leq n_0$.
\end{center}
\end{maintheorem}

To finish we state a direct corollary of our result.
Recall that a {\em sectional-Anosov flow} is a vector field whose maximal invariant set is sectional hyperbolic \cite{mem}.

\begin{clly}
\label{cor1}
For every sectional-Anosov flow
of a compact manifold there are
a neighborhood $ {\cal U}$ and $n_0 \in \mathbb{N}$ such that
\begin{center}
$\#\{$ $L$ is an attractor or repeller of $Y\in {\cal U}\} \leq n_0$.
\end{center}\end{clly}


\section{Proof}
\label{sech}

\noindent
An useful property of sectional hyperbolic sets
is given below.

\begin{lemma}
\label{l1}
Let $X$ be a $C^1$ vector field
of a compact $n$-manifold $M$, $n\geq3$, $X \in {\cal X}^1(M)$. 
Let $\Lambda \subset M$ be a sectional hyperbolic set of $X$.
Then, there is a neighborhood $ {\cal U} \subset {\cal X}^1(M)$ of $X$ 
and a neighborhood $U \subset M$ of $\Lambda$ such that if $Y \in {\cal U}$,
every nonempty, compact, non singular, invariant set $H$
of $Y$ in $U$ is hyperbolic {\em saddle-type} (i.e. $E^s\neq 0$ and $E^u\neq 0$).
\end{lemma}
\begin{proof}
See (\cite{mpp2}).
\end{proof}

This following theorem examinating the sectional hyperbolic 
splitting $T_{\Lambda}M = E^s_{\Lambda} \oplus E^c_{\Lambda}$
of a sectional hyperbolic set $\Lambda$ of $X \in {\cal X}^1(M)$
appears in \cite{lec}.

\begin{thm}
\label{t1}
Let $X$ be a $C^1$ vector field
of a compact $n$-manifold $M$, $n\geq3$, $X \in {\cal X}^1(M)$. 
Let $\Lambda \subset M$ be a sectional hyperbolic set of $X$.
If $\sigma\in Sing(X)\cap \Lambda$, then
$\Lambda \cap W^{ss}_X(\sigma)=\{\sigma\}.$
\end{thm}

We use it to prove the following.

\begin{prop}
\label{Ladilla1}
Let $X$ be a $C^1$ vector field
of a compact $n$-manifold $M$, $n\geq3$, $X \in {\cal X}^1(M)$. 
Let $\Lambda \subset M$ be a sectional hyperbolic set of $X$.
Let $\sigma$ be a singularity of $X$ in $\Lambda$.
Then, for every isolating block $U$ of $\Lambda$, there is a neighborhood $V$
of $W^{ss}(\sigma) \setminus \left\{ \sigma \right\}$ in $U$
such that
$$\left(\cap_{t>0}Y_t(U)\right)\cap V= \emptyset,$$
for every $C^1$ vector field $Y$
close to $X$.
\end{prop}
\begin{proof}
The equality in Theorem
\ref{t1} implies
that the negative
orbit of every point in
$W^{ss}_X(\sigma)\setminus \{\sigma\}$ leaves
$\Lambda$. 
Hence we can arrange neighborhood $V$ containing  
$W^{ss}(\sigma) \setminus \left\{ \sigma \right\}$ 
and such that
$$\Lambda \cap V= \emptyset$$
Since
$U$ is the isolating block of
$\Lambda$ we can find $T>0$ such that
$$X_T(U)\cap V =\emptyset.$$
Hence
$$Y_T(U)\cap V=\emptyset,$$
for all $C^r$ vector field
close to $X$.
The result follows
since
$\cap_{t>0}Y_t(U)\subset Y_T(U)$.
\end{proof}
\bigskip


Next we recall the standard definition of hyperbolic set.

\begin{defi}
\label{hyperbolic}
A compact invariant set $\Lambda$ of $X$ is {\em
hyperbolic}
if there are a continuous tangent bundle
invariant decomposition
$T_{\Lambda}M=E^s\oplus E^X\oplus E^u$ and positive constants
$C,\lambda$ such that

\begin{itemize}
\item $E^X$ is the vector field's
direction over $\Lambda$.
\item $E^s$ is {\em contracting}, i.e.,
$
\mid\mid DX_t(x) \left|_{E^s_x}\right.\mid\mid
\leq Ce^{-\lambda t}$, 
for all $x \in \Lambda$ and $t>0$.
\item $E^u$ is {\em expanding}, i.e.,
$
\mid\mid DX_{-t}(x) \left|_{E^u_x}\right.\mid\mid
\leq Ce^{-\lambda t},
$
for all $x\in \Lambda$ and $t> 0$.
\end{itemize}
A closed orbit is hyperbolic if it is also hyperbolic, as a compact invariant set. An attractor is hyperbolic if it is also a hyperbolic
set. 
\end{defi}

It follows from the stable manifold theory \cite{hps} that if $p$ belongs to a hyperbolic set $\Lambda$, then the following sets

\begin{tabular}{lll}
$W^{ss}_X(p)$ & = & $\{x:d(X_t(x),X_t(p))\to 0, t\to \infty\},$ \\
$W^{uu}_X(p)$ & = & $\{x:d(X_t(x),X_t(p))\to 0, t\to -\infty\},$ \\
\end{tabular}\\
are $C^1$ immersed submanifolds of $M$ which are tangent at $p$ to the subspaces $E^s_p$ and $E^u_p$ of $T_pM$ respectively.
Similarly,

\begin{tabular}{lll}
$W^{s}_X(p)$ & = & $ \bigcup_{t\in I\!\! R}W^{ss}_X(X_t(p))$, \\
$W^{u}_X(p)$ & = & $ \bigcup_{t\in I\!\! R}W^{uu}_X(X_t(p))$. \\
\end{tabular}\\
are also $C^1$ immersed submanifolds tangent to $E^s_p\oplus E^X_p$ and $E^X_p\oplus E^u_p$ at $p$ respectively.
Moreover, for every $\epsilon>0$ we have that

\begin{tabular}{lll}
$W^{ss}_X(p,\epsilon)$ & = & $\{x:d(X_t(x),X_t(p))\leq\epsilon, \forall t\geq 0\},$ and, \\
$W^{uu}_X(p,\epsilon)$ & = & $\{x:d(X_t(x),X_t(p))\leq \epsilon, \forall t\leq 0\}$\\
\end{tabular}\\
are closed neighborhoods of $p$ in $W^{ss}_X(p)$ and $W^{uu}_X(p)$ respectively.\\

Let $O = \left\{X_t (x): t \in \re \right\}$ be the orbit of $X$ through $x$,
then the stable and unstable manifolds of $O$ defined by
\begin{center}
	$W^s(O) =\cup_{x\in O} W^{ss}(x)$, and $W^u(O) =\cup_{x\in O} W^{uu}(x)$
\end{center}
are $C^1$ submanifolds tangent to the subbundles $E^s_{\Lambda} \oplus E^X_{\Lambda}$ 
and $E^X_{\Lambda} \oplus E^u_{\Lambda}$ respectively.

A {\em homoclinic orbit} of a hyperbolic periodic orbit $O$ is an orbit in $\gamma \subset W^s(O) \cap W^u(O)$.
If additionally $T_qM = T_qW^s(O) + T_qW^u(O)$ for some (and hence all) point $q \in \gamma$, then we say that
$\gamma$ is a {\em transverse homoclinic orbit} of $O$.  

\begin{defi}
The {\em homoclinic class} $H(O)$ of a hyperbolic periodic orbit $O$ is
the closure of the union of the transverse homoclinic orbits of $O$.
We say that an invariant set $L$ is a {\em homoclinic class} 
if $L = H(O)$ for some hyperbolic periodic orbit $O$.
\end{defi}

We denote by:\\ 
\begin{tabular}{l}
$Sing(X)$ the set of singularities of $X$.\\
$Cl(A)$ the closure of $A$, $A\subset M$.\\
If $\delta>0$, $B_\delta(A)=\{x\in M:d(x,A)<\delta\}$, 
where $d(\cdot,\cdot)$ is the metric in $M$.\\
\end{tabular}\\

\begin{lemma}
\label{le2}
Let $X$ be a $C^1$ vector field
of a compact $n$-manifold $M$, $X \in {\cal X}^1(M)$. 
Let $\Lambda \in M$ be a hyperbolic set of $X$.
Then, there is a neighborhood $ {\cal U} \subset {\cal X}^1(M)$ of $X$, 
a neighborhood $U \subset M$ of $\Lambda$ and $n_0 \in \mathbb{N} $ such that
\begin{center}
	 $\#\{L\subset U:$ $L$ is homoclinic class of $Y\in {\cal U}\} \leq n_0$
\end{center}
for every vector field $Y \in {\cal U}$.
\end{lemma}

\begin{proof}
By the stability of hyperbolic sets we can fix a neighborhood 
$U \subset M$ of $\Lambda$, a neighborhood ${\cal U} \subset {\cal X}^1(M)$ 
of $X$ and and $\epsilon>0$ such that every hyperbolic set $H\subset U$ of 
every $Y\in \mathcal{U}$ satisfies that 
\begin{equation}
	\begin{tabular}{l}
		$W^{ss}_Y(x,\epsilon)$, $W^{uu}_Y(x,\epsilon)$ have uniform size $\epsilon$
		for all $x \in H$\\
	\end{tabular}
\label{El2}
\end{equation}

By contradiction, we suppose that
there exists a sequence of vector fields $X^n\in\mathcal{U}$ converging to $X$ such that
\begin{center}
	 $\#\{L\subset U:$ $L$ is homoclinic class of $X^n\} \geq n$
\end{center}

It is well known \cite{haka} that the periodic orbits 
are dense in $L^n \subset \Lambda^n=\Lambda_{X^n}$, for all $n \in \nat$. 
Moreover, these homoclinic classes are pairwise disjoint. 

Let $\epsilon>0$ be the uniform size by (\ref{El2}), and let $\eta>0$ be 
such that $0< \eta < \frac{\epsilon}{2}$.

Since $U$ is neighborhood of $\Lambda$, $Cl(U)$ is compact 
neighborhood of $\Lambda$, then we can cover $Cl(U)$ 
with a finite number of balls with radius $\frac{\eta}{2}$.
We denote this finite number by $n_0$.

Thus, if two periodic points $p_1,p_2 \in L$ satisfies $d(p_1,p_2)< \eta$, then 
\begin{equation}
W^{ss}_X(p_1,\epsilon) \cap W^{uu}_X(p_2,\epsilon) \neq \emptyset
\label{ch}
\end{equation}

Therefor, for every vector field $X^{N}$ with $N>n_0$, 
we have that there are homoclinic classes 
$L^i,L^j$ of $X^{N}$ in $Cl(U)$ 
contained in the same ball with radius $\frac{\eta}{2}$, $1\geq i < j \geq N$.

Since $L^i$ and $L^j$ are homoclinic classes, there are periodic points 
$p^i$ and $p^j$ of $L^i$ and $L^j$ respectively satisfying (\ref{ch}), 
then $p^i$ and $p^j$ belongs to the same homoclinic class and this imply $L^i = L^j$.
Thus, the sequence $(L^n)_{n \in \nat}$, is constant for $n$ enough large. 
This is a contradiction and the proof follows.
\end{proof}

\begin{lemma}
\label{le3}
Let $X$ be a $C^1$ vector field
of a compact $n$-manifold $M$, $n\geq3$, $X \in {\cal X}^1(M)$. 
Let $\Lambda \in M$ be a sectional hyperbolic set of $X$.
Let $Y^n$ be a sequence of vector fields
converging to $X$ in the $C^1$ topology.
There is a neighborhood $U \subset M$ of $\Lambda$, 
such that if $R^n$ is a repeller of $Y^n$, $R^n \subset \cap_{t>0}Y_t^n(U)$ 
for each $n \in \mathbb{N}$, then the sequence $(R^n)_{n \in \nat}$ of repellers do 
not accumulate on the singularities of $X$, i.e., 
$$Sing(X) \bigcap Cl(\cup_{n \in \nat} R^n) = \emptyset $$
\end{lemma}

\begin{proof}
Let $\sigma \in Sing(X)$ and we denote $\Lambda_Y=\cap_{t>0}Y_t^n(U)$. 
Fix the neighborhood $U$ of $\Lambda$ as in Lemma \ref{l1} and thus
we can assume that $U$ is an isolating block of $\Lambda$. 
Assume by contradiction that 
$$Sing(X) \bigcap Cl(\cup_{n \in \nat} R^n) \neq \emptyset. $$
Then, exist a sequence $(x_n)_{n\in\nat}$, with $x_n \in R^n \subset \Lambda_Y$, for all 
$n \in \nat$, and such that $x_n \longrightarrow \sigma$. 

Since $\Lambda$ is sectional hyperbolic set, we have (by Theorem \ref{t1}) that 
$\Lambda \cap W^{ss}_X(\sigma) = \left\{ \sigma \right\}$, 
and as $Y^n \longrightarrow X$ 
(by Proposition \ref{Ladilla1}), there is a neighborhood $V$
of $W^{ss}(\sigma) \setminus \left\{ \sigma \right\}$ in $M$
such that $\Lambda_Y\cap V= \emptyset$,
for $n \in \nat$ large enough.
 
As $x_n \longrightarrow \sigma$, for $\epsilon>0$ uniform size, 
$W^{ss}_{Y^n}(x^n, \epsilon) \longrightarrow W^{ss}_{X}(\sigma, \epsilon)$ 
in the sense $C^1$ manifolds \cite{pt}.

Then, for $n \in \nat$ enough large, 
$W^{ss}_{Y^n}(x^n, \epsilon) \cap V \neq \emptyset$.
Note that $W^{ss}_{Y^n}(x^n, \epsilon) \subset W^{ss}_{Y^n}(x^n) \subset R^n$, 
since is repeller of $Y^n$.
Hence $R^n \cap V \neq \emptyset$, then $\Lambda_Y \cap V \neq \emptyset$.
This is a contradiction.
\end{proof}



Let $M$ be a compact $n$-manifold, $n \geq 3$.
Fix $X\in {\cal X}^1(M)$, inwardly
transverse to the boundary $\partial M$. We denotes
by $X_t$ the flow of $X$, $t\in I\!\! R$.\\

There is also a stable manifold theorem in the case when $\Lambda$ is sectional hyperbolic set.
Indeed, denoting by $T_{\Lambda}M=E^s_{\Lambda}\oplus E^c_{\Lambda}$ the corresponding the sectional hyperbolic
splitting over $\Lambda$ we have from \cite{hps} that
the contracting subbundle $E^s_{\Lambda}$
can be extended to a contracting subbundle $E^s_U$ in $M$. Moreover, such 
an extension is tangent to a continuous foliation denoted by $W^{ss}$ (or $W^{ss}_X$ to indicate dependence on $X$).
By adding the flow direction to $W^{ss}$ we obtain a continuous foliation 
$W^s$ (or $W^s_X$) now tangent to $E^s_U\oplus E^X_U$.
Unlike the Anosov case $W^s$ may have singularities, all of which being
the leaves $W^{ss}(\sigma)$ passing through the singularities $\sigma$ of $X$.
Note that $W^s$ is transverse to $\partial M$
because it contains the flow direction (which is transverse to $\partial M$
by definition).

It turns out that every singularity $\sigma$ of a sectional hyperbolic set 
$\Lambda$ satisfies $W^{ss}_X(\sigma)\subset W^s_X(\sigma)$.
Furthermore, there are two possibilities for such a singularity, namely,
either $dim(W^{ss}_X(\sigma))=dim(W^s_X(\sigma))$ (and so $W^{ss}_X(\sigma)=W^s_X(\sigma)$) or $dim(W^{s}_X(\sigma))=dim(W^{ss}_X(\sigma))+1$.
In the later case we call it Lorenz-like according to the following definition. 

\begin{defi}
\label{ll} 
We say that a singularity $\sigma$ of a sectional-Anosov flow $X$ is {\em Lorenz-like}
if $dim (W^s(\sigma))=dim (W^{ss}(\sigma))+1.$
\end{defi}

Let $\sigma$ be a singularity Lorenz-like of a sectional hyperbolic set $\Lambda$. 
We will denote $dim(W^{ss}_X(\sigma))=s$ and 
$dim(W^{u}_X(\sigma))=u$, 
therefore $\sigma$ has a $(s+1)$-dimensional 
local stable manifold $W^s_X(\sigma)$. 
Moreover $W^{ss}_X(\sigma)$ separates $W^s_{loc}(\sigma)$
in two connected components denoted by $W^{s,t}_{loc}(\sigma)$ 
and $W^{s,b}_{loc}(\sigma)$ respectively.

\begin{defi}
\label{n1}
A {\em singular-cross section} of a Lorenz-like singularity 
$\sigma$ will be a pair of submanifolds $\Sigma^t, \Sigma^b$, where 
$\Sigma^t, \Sigma^b$ are cross sections and;\\

\begin{tabular}{l}
$\Sigma^t$ is transversal to $W^{s,t}_{loc}(\sigma)$. \\
$\Sigma^b$ is transversal to $W^{s,b}_{loc}(\sigma)$.\\ 
\end{tabular}\\

Note that every singular-cross section
contains a pair singular submanifolds $l^t,l^b$ 
defined as the intersection of the local stable manifold of $\sigma$ with $\Sigma^t,\Sigma^b$ 
respectively.

Also note that $dim(l^*)=dim(W^{ss}(\sigma))$.\\

If $*=t,b$ then $\Sigma^*$ is a {\em hypercube of dimension $(n-1)$},   
i.e., diffeomorphic to $B^u[0,1] \times B^{ss}[0,1]$, with 
$B^u[0,1] \approx I^u$, $B^{ss}[0,1] \approx I^s$, $I^k=[-1,1]^k$, $k \in \mathbb{Z}$ and  where:\\

\begin{tabular}{l}
$B^{ss}[0,1]$ is a ball centered at zero and radius $1$ contained in $\re^{dim(W^{ss}(\sigma))}=\re^{s}$\\
$B^{u}[0,1]$ is a ball centered at zero and radius $1$ contained in $\re^{dim(W^{u}(\sigma))}=\re^{n-s-1}$
\end{tabular}\\

Let $f: B^u[0,1] \times B^{ss}[0,1] \longrightarrow \Sigma^*$ be the diffeomorphism, 
where $$f(\left\{0\right\} \times B^{ss}[0,1])=l^*$$ and $\left\{0\right\}=0 \in \re^u$.
Hence, we denoted the boundary of $\Sigma^*$ for $\partial \Sigma^*$,
and $\partial \Sigma^* = \partial^h \Sigma^* \cup \partial^v \Sigma^*$ such that\\

\begin{tabular}{l}
$\partial^h \Sigma^* = \left\{\right.$the union of the boundary submanifolds
 which are transverse to $l^*$ $\left.\right\}$\\
$\partial^v \Sigma^* = \left\{\right.$ the union of the
boundary submanifolds which are parallel to $l^*$ $\left.\right\}.$
\end{tabular}\\

Moreover,
$$\partial^h \Sigma^*  =  (I^u \times [\cup_{j=0}^{s-1} (I^j \times \left\{-1\right\} \times I^{s-j-1})])
\bigcup (I^u \times[\cup_{j=0}^{s-1} (I^j \times \left\{1\right\} \times I^{s-j-1})])$$
$$\partial^v \Sigma^*  =  ([\cup_{j=0}^{u-1} (I^j \times \left\{-1\right\} \times I^{u-j-1})] \times I^s)
\bigcup ([\cup_{j=0}^{u-1} (I^j \times \left\{1\right\} \times I^{u-j-1})] \times I^s)$$
and where $I^0 \times I=I$.
\end{defi}

Hereafter we denote $\Sigma^* = B^u[0,1] \times B^{ss}[0,1]$.


\begin{prop}
\label{le4}
Let $X$ be a $C^1$ vector field
of a compact $n$-manifold $M$, $n\geq3$, $X \in {\cal X}^1(M)$. 
Let $\Lambda \subset M$ be a sectional hyperbolic set of $X$.
Then, there are neighborhoods ${\cal U}$ of $X$, 
$U$ of $\Lambda$ and $n_0\in \mathbb{N}$ such that 
\begin{center}
$\#\{A\subset U:$ $A$ is an attractor of $Y\in {\cal U}\} \leq n_0$.
\end{center}
\end{prop}

\begin{proof}
The proof is by contradiction, i.e., suppose that 
for $n \in \nat$, we have that for all neighborhood ${\cal U}$ of $X$, 
exists $Y\in {\cal U}$ such that 
\begin{center}
	 $\#\{A\subset U:$ $A$ is an attractor of $Y\in {\cal U}\} \geq n$.
\end{center}
Then, there is a sequence of vectors fields $(X^n)_n \in \nat$, 
such that $X^n \stackrel{C^1}{\rightarrow} X$, and a sequence $(A^n)_n \in \nat$, 
where $A^n$ is an attractor of vector field $X^n$, for all $n$. 
By compactness we can suppose that the attractors are non-singular, since the 
singularities are isolated. 
Fix the neighborhood $U$ of $\Lambda$ as in Lemma \ref{l1} and thus
we can assume that $U$ is an isolating block of $\Lambda$.

We claim that the sequence $(A^n)_{n \in \nat}$ of attractors 
accumulate on the singularities of $X$, otherwise 
$Sing(X) \bigcap Cl(\cup_{n \in \nat} A^n) = \emptyset$, 
then, there is $\delta>0$, such that 
$B_\delta(Sing(X))\bigcap\left(\cup_{n\in I\!\! N}A^n\right)=\emptyset.$	

Thus, in the same way as in \cite{a}, we define
\begin{equation}
H=\cap_{t\in I\!\! R}X_t\left(U\setminus B_{\delta/2}(Sing(X))\right)
\label{bola}
\end{equation}

By definition $Sing(X)\cap H=\emptyset$, $H$ is compact since $\Lambda$ is, 
and $H$ is a nonempty compact set \cite{a}, which is clearly invariant for $X$.
It follows that $H$ is hyperbolic by Lemma \ref{l1} and by Lemma \ref{le2} there is 
$n_0 \in \nat$ such that the sequence of attractors is bounded by $n_0$, that is a contradiction.\\

Then, the sequence $(A^n)_{n \in \nat}$ of attractors 
accumulate on the singularities of $X$, i.e., $Sing(X) \bigcap Cl(\cup_{n \in \nat} A^n) \neq \emptyset$. 
Thus, exists $\sigma\in U$ such that 
$\sigma\in Sing(X)\bigcap Cl\left(\cup_{n\in I\!\! N}A^n\right).$

The subbundle $E^s$ of $\Lambda$ extends to a contracting invariant subbundle on the whole
$U$ and we take a continuous (not necessarily invariant) extension of $E^c$ in $U$. We have that 
this extension persists by small perturbations of $X$ \cite{hps} and we denote the 
splitting by $E^{s,n} \oplus E^{c,n}$, where $E^{s,0} \oplus E^{c,0} = E^s \oplus E^c$.
We can assume that $\sigma(X^n)=\sigma$ and $l^t\cup l^b\subset W^s_{X^n}(\sigma)$
for all $n$.

As before we fix a coordinate system $(x,y)=(x^*,y^*)$
in $\Sigma^*$ with ($*=t,b$) and such that 
$\Sigma^*= B^u[0,1] \times B^{ss}[0,1]$ and $l^*=\{0\}\times B^{ss}[0,1]$ 
with respect to $(x,y)$.

Denote by $\Pi^*:\Sigma^*\to B^u[0,1]$ the projection, where $\Pi^*(x,y)=x$ and 
for $\Delta>0$ we define $\Sigma^{*,\Delta}= B^u[0,\Delta] \times B^{ss}[0,1].$

Then, by Theorem \ref{t1} we have that $\Lambda \cap W^{ss}_X(\sigma)=\{\sigma\}$ and by 
Lemma \ref{l1} $A^n$ is a hyperbolic attractor of type saddle
of $X^n$ for all $n$. Then by \cite{a} 
for every isolating block $U$ of $\Lambda$ we can choose 
$\Sigma^t,\Sigma^b$, singular-cross section for $\sigma$ in $U$ such that
\begin{equation}
\label{int}
(\cap_{t>0}X^n_t(U))\cap\left(\partial^h\Sigma^t\cup\partial^h\Sigma^b\right)=\emptyset
\end{equation} 
and we have that there is $n_1$ such that $A^{n_1}\cap int(\Sigma^{*,\Delta_0})\neq\emptyset$.

We shall assume that $A^{n_1}\cap int(\Sigma^{t,\Delta_0})\neq\emptyset$
(Analogous proof for the case $*=b$).
By (\ref{int}) we have $A^{n_1}\cap \partial^h\Sigma^{t,\Delta_0}=\emptyset.$
and by compactness we have that there is $p\in \Sigma^{t,\Delta_0} \cap A^{n_1}$ such that
\begin{center}
$ dist(\Pi^t(\Sigma^{t,\Delta_0}\cap A^{n_1}),0)=dist(\Pi^t(p),0),$	
\end{center}
where $dist$ denotes the distance in $B^u[0,\Delta_0]$. 
Note that $dist(\Pi^t(p),0)$ is the minimum distance of
$\Pi^t(\Sigma^{t,\Delta_0} \cap A^{n_1})$ to $0$ in $B^u[0,\Delta_0]$.\\

As $W^u_{X^{n_0}}(p) \subset A^{n_1}$, since $A^{n_1}$ is attractor, 
we have that $W^u_{X^{n_1}}(p)\cap \Sigma^{t,\Delta_0}$ contains some compact manifold $K^{n_1}$. 

We have that
$K^{n_1}$ is {\em transverse} to $\Pi^t$
(i.e. $K^{n_1}$ is transverse to the curves
$(\Pi^t)^{-1}(c)$, for every $c \in B^u[0,\Delta_0]$)(See \cite{d},\cite{a}). 
First we denote $\Pi^t(K^{n_1})=K^{n_1}_1$ 
the image of $K^{n_0}$ by the projection $\Pi^t$ in $B^u[0,\Delta_0]$. 
Note that $K^{n_1}_1\subset B^u[0,\Delta_0]$ and 
$\Pi^t(p)\in int(K^{n_1}_1)$.

Since $dim(K^{n_1}_1)=dim(B^u[0,\Delta_0])=(n-s-1)$,
there is  $z_0\in K^{n_0}$ such that
\begin{center}
	$dist(\Pi^t(z_0),0)<dist(\Pi^t(p),0).$
\end{center}

As $A^{n_1}\cap \partial^h\Sigma^{t,\Delta_0}=\emptyset$
(\ref{int}), $K^{n_1}\subset W^u_{X^{n_1}}(p)$ and 
$dim(K^{n_1}_1)=dim(B^u[0,\Delta_0])$,
we conclude that $dist(\Pi^t(\Sigma^{t,\Delta_0}\cap A^{n_1}),0)=0$,
and this last equality implies that
\begin{center}
	$A^{n_1}\cap l^t\neq\emptyset.$
\end{center}
Since $l^t\subset W^s_{X^{n_1}}(\sigma)$
and $A^{n_1}$ is closed invariant set for $X^{n_1}$
we conclude that $\sigma\in A^{n_1}$.
This is a contradiction, since by hypotheses we have that 
$A^n$ is non-singular for all $n\in \nat$ and the proof follows.

\end{proof}



\begin{proof}[Proof of Theorem \ref{thB}]

We prove the theorem by contradiction.
Let $X$ be a $C^1$ vector field
of a compact $n$-manifold $M$, $n\geq3$, $X \in {\cal X}^1(M)$. 
Let $\Lambda \in M$ be a sectional hyperbolic set of $X$.
Then, we suppose that
there is a sequence $(X^n)_{n\in \nat} \subset {\cal X}^1(M)$, 
$X^n\overset{C^1}{\to} X$ such that every $X^n$ exhibits 
$n$ attractors or repellers, with $n>n_0$. By Proposition \ref{le4} 
there is a neighborhood $ {\cal U} \subset {\cal X}^1(M)$ of $X$ 
and a neighborhood $U \subset M$ of $\Lambda$
such that the attractors in $U$ are finite for all $Y \in {\cal U}$.
Thus, we are left to prove only for the repeller case.
We denote by $R^n$ a repeller of $X^n$ in $\cap_{t>0}X_t^n(U)=\Lambda_{X^n}$. 
Since $\Lambda_{X^n}$ arbitrarily close to $\Lambda$ and since $R^n \in \Lambda_{X^n}$, 
$R^n$ also is arbitrarily close to $\Lambda$, 
we can assume that $L^n$ belongs to $\Lambda$ for all $n$.

Let $(R^n)_{n\in \nat}$ be the sequence of  
repellers contained in $\Lambda$. By the Lemma \ref{le3} we have that 
\begin{center}
	$Sing(X) \bigcap Cl(\cup_{n \in \nat} R^n) = \emptyset$
\end{center}
Then, we have that there is $\delta>0$, such that 
$B_\delta(Sing(X))\bigcap\left(\cup_{n\in \nat} R^n\right)=\emptyset.$

As in (\ref{bola}) we define
$H=\bigcap_{t\in I\!\! R}X_t\left(U\setminus B_{\delta/2}(Sing(X))\right)$.
It follows that $H$ is hyperbolic by Lemma \ref{l1}
and by the Lemma \ref{le2} we have that 
there is a neighborhood $ {\cal U} \subset {\cal X}^1(M)$ of $X$, 
a neighborhood $U \subset M$ of $H$, and $n_1 \in \mathbb{N} $ such that
\begin{center}
	 $\#\{R\subset U:$ $R$ is a repeller of $Y\in {\cal U}\} \leq n_1 \leq n_0$
\end{center}
for every vector field $Y \in {\cal U}$. This is a contradiction, since by hypotheses 
we have that
\begin{center}
$\#\{R\subset H:$ $R$ is a repeller of $Y\in {\cal U}\} \geq n>n_0$.
\end{center}

\end{proof}

\begin{center}
ACKNOWLEDGEMENTS	
\end{center}
The author would like to thank B. Santiago by useful conversations.


\medskip 

\flushleft
A. M. L\'opez B\\
Instituto de Matem\'atica, Universidade Federal do Rio de Janeiro\\
Rio de Janeiro, Brazil\\
E-mail: barragan@im.ufrj.br


\begin{thebibliography}{HHH9}





\bibitem{ams}
Arbieto, A., Morales, C.A., Senos, L.,
On the sensitivity of sectional-Anosov flows,
{\em Math. Z} 270 (2012), no. 1-2, 545--557. 



\bibitem{ap}
Araújo, V., Pacifico, M.J.,
Three-dimensional flows. With a foreword by Marcelo Viana. 
{\em Ergebnisse der Mathematik und ihrer Grenzgebiete.
3. Folge. A Series of Modern Surveys in Mathematics [Results in Mathematics and Related Areas. 3rd Series. A Series of Modern Surveys in Mathematics], 53. Springer, Heidelberg,} 2010.


\bibitem{lec}
Bautista, S., Morales, C.A.,
Lectures on sectional-Anosov flows,
{\em http : www:preprint:impa:br=Shadows=SERIED=2011=86:html.}


\bibitem{d}
Doering, C. I.,
Persistently transitive flows on three-dimensional manifolds,
{\em Dynamical systems and bifurcation theory (Rio de Janeiro, 1985), 
Pitman Res. Notes Math. Ser.} {\bf 160} (1987), 
Longman Sci. Tech., Harlow, 59-89.


\bibitem{haka}
Hasselblatt, B., Katok, A.,
Introduction to the modern theory of dynamical systems,
{\em Cambridge university press}, Vol. 54, (1997).

\bibitem{hps}
Hirsch, M., Pugh, C., Shub M.,
Invariant manifolds,
Lec. Not. in Math. {\bf 583} (1977), 
{\em Springer-Verlag}.

\bibitem{a}
Lopez, A., M.,  
Sectional Anosov flows in higher dimensions,
{\em http://arxiv.org/abs/1308.6597}.

\bibitem{mem}
Metzger, R., Morales, C.A., 
Sectional-hyperbolic systems,
{\em Ergodic Theory Dynam. Systems} 28 (2008), no. 5, 1587-1597.



\bibitem{m}
Morales, C., A.,  
The explosion of singular-hyperbolic attractors,
{\em Ergodic Theory Dynam. Systems} 24 (2004), no. 2, 577-591.




\bibitem{mpp2}
Morales, C. A., Pac\'{\i}fico, M. J., Pujals, E.R.,
Singular Hyperbolic Systems,
{\em Proc. Amer. Math. Soc.}
{\bf 127} (1999), 3393-3401.


\bibitem{pt}
Palis, J., Takens, F.,
Hyperbolicity and sensitive chaotic dynamics at 
homoclinic bifurcations (1993),
{\em Cambridge Univ. Press.}


\end{thebibliography}
\end{document}